\pdfoutput=1

\documentclass[12pt]{amsart}
\usepackage{amsmath,amscd,amsfonts,amssymb,latexsym}

\usepackage{tikz}
\usepackage{enumerate}
\def\hbar{h}

\newtheorem{theorem}{Theorem}[section]
\newtheorem{proposition}[theorem]{Proposition}
\newtheorem{lemma}[theorem]{Lemma}
\newtheorem{corollary}[theorem]{Corollary}
\newtheorem{remark}[theorem]{Remark}
\newtheorem{definition}[theorem]{Definition}
\newtheorem{example}[theorem]{Example}

\newtheorem*{corollary*}{Corollary}
\newtheorem*{theorem*}{Theorem}

\def\C{\mathbb{C}}
\def\Q{\mathbb{Q}}
\def\T{{\mathbb{A}}}
\def\TT{{\mathbb{T}}}
\def\tt{{\bf t}}

\def\R{\mathbb{R}}
\def\Z{\mathbb{Z}}
\def\P{\mathbb{P}}
\def\Hom{{\rm Hom}}

\def\O{{\mathcal  O}}

\def\L{{\mathcal L}}

\def\vt{\vartheta}

\def\tt{{\mathbf t}}
\def\ii{{i}}
\def\mC{{\rm mC}}
\def\rk{{\rm rk}}

\def\D{{\widetilde{\partial Y}}}
\def\K{{\mathcal K}}
\DeclareMathOperator\Ne{{\mathcal N}}

\def\ee{{\bf e}}
\def\rank{\operatorname{rank}}
\def\we{{\rm w}}

\def\iee{i}
\def\Ell{{\mathcal E\ell\ell}}

\author{Jakub Koncki}
\address{Institute of Mathematics, University of Warsaw, Poland}
\email{j.koncki@mimuw.edu.pl}

\author{Andrzej Weber}
\address{Institute of Mathematics, University of Warsaw, Poland}
\email{aweber@mimuw.edu.pl}

\thanks{Both authors are supported by NCN grant  2016/23/G/ST1/04282 (Beethoven 2)}

\title[Twisted motivic Chern class and stable envelopes]
{\large Twisted motivic Chern class\\ and stable envelopes}

\begin{document}
\begin{abstract}We present a  definition of {\em twisted motivic Chern classes} for singular pairs $(X,\Delta)$ consisting of a singular space $X$ and a $\Q$-Cartier divisor containing the singularities of $X$.  The definition  is a mixture of  the construction of motivic Chern classes previously defined by Brasselet-Sch{\"u}rmann-Yokura with the construction of multiplier ideals. The twisted motivic Chern classes are the limits of the elliptic classes defined by Borisov-Libgober. We show that with a suitable choice of the divisor $\Delta$ the twisted motivic Chern classes satisfy the axioms of the stable envelopes in the K-theory. Our construction is an extension of the results proven by the first author for the fundamental slope.
\medskip

\noindent {\it Keywords}:
Twisted motivic Chern class; stable envelope; Bia{\l}ynicki-Birula decomposition; torus localization.
\end{abstract}
\maketitle

\baselineskip 15pt

Algebraic torus action on a complex variety induces a decomposition into attracting sets, called Bia{\l}ynicki-Birula cells. The closures of the Bia{\l}ynicki-Birula cells may be singular and their properties, in particular their cohomological invariants, were studied recently by many authors. However not much is known beyond the classical examples, such as the Schubert varieties in the homogeneous spaces. We would like to present some general statements connecting Okounkov stable envelopes with a version of motivic Chern classes.
{The notion of stable envelope is evolving. Initially it was
 defined for torus actions on symplectic manifolds, under assumption that the cells satisfy certain regularity conditions (see \cite{O3} for a recent development in which manifold is not assumed symplectic).}
 The definition is axiomatic and only  particular examples were studied in detail.
We will consider the case when the symplectic manifold is the cotangent bundle of a smooth variety on which the torus acts with finitely many fixed points.  In that context we construct  the stable envelopes (for an arbitrary slope) via resolution of singularities of the closures of Bia{\l}ynicki-Birula cells.

\medskip The idea of connecting characteristic classes with stable envelopes originates from parallelly written papers \cite{RV} or \cite{FR} and \cite{AMSS0}. There, the Chern-Schwartz-MacPherson classes in equivariant cohomology were compared with cohomological stable envelopes. Next,  the  K-theoretical stable envelopes were compared  with equivariant motivic Chern classes, \cite{FRW, AMSS}. It was done only for a special  class of spaces, in particular for the generalized flag varieties, and only for a special value of the {\it slope}, on which stable envelopes depend. One should mention also the work on weight functions, which apply for classical partial flag varieties $GL_n/P$, \cite{RTV}. Further came papers comparing elliptic stable envelopes with Borisov-Libgober elliptic classes. The work was done only for generalized flag varieties $G/P$, \cite{RW, KRW}. Now, motivated by the results in elliptic theory we come back to the K-theory, but we consider a general slope, without any further assumption on the variety, except that the torus acts with isolated fixed points. The paper is a continuation of \cite{K}, where only the fundamental slope was considered.
\medskip

We begin with the definition of the {\em twisted motivic Chern classes}. We assume that $X$ is a singular variety, $\Delta$ is a $\Q$-Cartier divisor on $X$ and that the support of $\Delta$ contains the singularities of $X$. The construction is a combination of the Brasselet-Sch{\"u}rmann-Yokura definition \cite{BSY} with the ideas coming from the theory of multiplier ideals, \cite[\S 9]{La}.  The twisted motivic Chern class can also be interpreted as a limit of the elliptic class constructed by Borisov and Libgober for Kawamata log-terminal pairs. Our classes, denoted by $\mC(X;\Delta)$  (and a relative version $\mC(X,\partial X;\Delta)$) are defined in the K-theory of coherent sheaves of $X$.  The definition is explicit in terms of the resolution of singularities. We apply the Weak Factorization Theorem to show that the resulting class is well defined. In the equivariant context, i.e.~when we assume a group action on $X$, we obtain classes in the equivariant K-theory.
\medskip

An important technical result of the paper is Theorem \ref{thm:main}. Assume that a torus $\T$ is acting on a smooth variety $M$ and $x\in M$ is an isolated fixed point. Let $X=\overline M^+_x$ be the closure of a Bia{\l}ynicki-Birula cell and let $i:X\hookrightarrow M$ be the inclusion. The result  says that the localized twisted motivic Chern classes satisfy the Newton inclusion property. Roughly speaking, for any boundary $\Q$-Cartier divisor $\Delta$ the restricted class $$i_*\mC(X,\partial X;\Delta)_x\in K_\T(\{x\})[y]$$
(which can be interpreted as a Laurent polynomial) has the exponents bounded by the exponents of the Euler class of the ambient space, and eventually shifted by a vector determined by the divisor.
\medskip

Finally we state our main result.
Let $M$ be a smooth projective manifold with a torus $\T$ action. Let us assume that the fixed point set $M^\T$ is finite. Let $$M=\bigsqcup_{\ee\in M^\T} M_\ee^+$$ be the decomposition into the Bia{\l}ynicki-Birula cells.
The twisted motivic Chern classes applied to the closures of the  cells $M_\ee^+\subset M$, with a suitably chosen divisors $\Delta_{\ee,s}$ satisfy two of three axioms of Okounkov stable envelopes without any additional assumptions.

\begin{theorem*}[see Theorem \ref{tw:podsumowanie}]
	The (rescaled) classes
	$$i_*\mC(\overline M^*_\ee,\partial M^*_\ee;\Delta_{\ee,s})\in K_\T(M)[y]\subset K_{\T\times\C^*}(T^*M)$$
	{satisfy} the normalization axiom and the Newton inclusion property of stable envelopes.\end{theorem*}

 The remaining axiom -- the support axiom -- is of a different nature. Its validity depends on the regularity of the Bia{\l}ynicki-Birula decomposition and it does not hold automatically. To even begin we have to assume that the BB-decomposition is a stratification satisfying the Whitney condition A. The stratification by Schubert cells in the homogeneous space $G/P$, where $G$ is a semisimple group and $P$ a parabolic subgroup, is regular enough. It is locally of a product form. Therefore we obtain

\begin{corollary*}
	Consider a homogeneous variety $G/P$ with the action of the maximal torus $\T$. The twisted motivic Chern   classes
		$$i_*\mC(\overline{G/P^+_{w}},\partial (G/P^+)_{w};\Delta_{w,s})\in K_\T(G/P)[y]\subset K_{\T\times\C^*}(T^*G/P)$$
after a suitable rescaling satisfy the axioms of the stable envelopes.
\end{corollary*}

We would like to stress that the identification of twisted motivic Chern classes as K-theoretic stable envelopes can be applied in much more general situations.
\bigskip

We would like to thank J{\"o}rg Sch{\"u}rmann for valuable conversations and for
various important remarks.

\tableofcontents

\section{Notation}
\label{notacja}
\subsection{Weights and Newton polytopes}
\label{notacja1}

We assume that the spaces considered by us are equipped with a torus $\T\simeq (\C^*)^{\dim(\T)}$ action. The weights of the torus action will play a fundamental role in our consideration.
\medskip

Let $\L\to X$ be an equivariant line bundle i.e.~a line bundle together with a linearization. Suppose that $p\in X$ is a fixed point. Then the fiber $\L_{p}$ is a representation of $\T$. By $\we_p(\L)\in \Hom(\T,\C^*)\otimes\Q$ we denote the weight of this representation.
The weight is a locally constant function on $X^\T$, that is if $p$ and $q$ are two points belonging to the same component of $X^\T$, then $\we_p(\L)=\we_q(\L)$ and it makes sense to define $\we_F(\L)$ for a component $F\subset X^\T$.
	{Let $D$ be an equivariant Cartier divisor in $X$. It induces an element in the equivariant Chow group $A^{\T}_{\dim X-1}(X)$. By \cite[Theorem 1]{EG} the first Chern class induces an isomorphism $Pic^\T(X)\simeq A^{\T}_{\dim X-1}(X)$. Therefore, the line bundle $\O_X(D)$ is equipped with the natural choice of linearization.
	For a smooth fixed point $p\in D^\T$ the weight $\we_p(\O_X(D))$ is the normal weight to $D$ at $p$. For a fixed point $p$ which does not belong to the support of $D$, we have $\we_p(D)=0$.}
	Let
	$$\we_p(D)=\we_p(\O_X(D))\,,$$
	where $\O_X(D)$ is considered with the natural linearization.  We extend the notion of weights to $\Q$-Cartier divisors:
if $mD$ is integral, then $$\we_p(D)=\frac1m\we_p(\O_X(mD))\,.$$
\medskip

For a weight $w\in  \Hom(\T,\C^*)$ we use additive notation. The monomial $t^w$ is understood as a one dimensional representation of $\T$. For a fractional weight $w\in \Hom(\T,\C^*)\otimes\Q$ the symbol $t^w$ denotes a one dimensional representation of a suitable cover $\widetilde \T$ of the torus $\T$. For a given element $\alpha \in K_\T(X)$ the product $t^w\alpha$ makes sense in $K_{\widetilde \T}(X)$.
\medskip

Suppose the torus $\T$ acts trivially on $F$. Then $K_\T(F)\simeq K(F)\otimes \Hom(\T,\C^*)$. Let $\alpha\in K_\T(F)$. The element $\alpha$ can be uniquely written as $\sum_{i=1}^k t^{w_i}\alpha_i$ with all $w_i$ distinct and $\alpha_i\in K(F)$ nonzero. Then the Newton polytope of $\alpha$ is defined as the convex hull
$$\Ne^{ \T}(\alpha)=conv(w_i\;|\;i=1,2,\dots, k)\,.$$
We have
$$\Ne^{ \T}(t^w\alpha)=\Ne^{ \T}(\alpha)+w$$
for $\alpha\in K_\T(F)$, $w\in \Hom(\T,\C^*)\otimes\Q$.

\subsection{K-theory classes}

We consider spaces with a torus $\T\simeq(\C^*)^{rk(\T)}$ action. For the sections \S\ref{sec:def} and \S\ref{sec:niez} we can assume that a more general linear algebraic group $G$ is acting.
By $\K(-)$ we mean the equivariant K-theory (topological or algebraic) {of vector bundles} with the additional variable $y$ adjoint
$$\K(M):=K_G(M)[y]\,.$$
The K-theoretic Euler class of a complex vector bundle $E$ is defined as
$$eu(E)=\lambda_{-1}(E^*)=\sum_{k=0}^{{\rm rk}(E)}(-1)^k\Lambda^k E^*\,.$$
\medskip

The motivic Chern class for a map of varieties $f:X\to M$ was defined in \cite{BSY}. When we assume that $M$ is a smooth variety, then $$\mC(f:X\to M)\in \K(M)\,.$$
We refer the reader to \cite{FRW, AMSS} for the definition in the equivariant context. For a subvariety $X\subset M$ we set
$\mC(X\subset M)=\mC(\iota: X\to M)$, where $\iota:X\to M$ is the inclusion.
It is possible to consider the  motivic Chern class in the K-theory of coherent sheaves on $X$, but for our purposes it is enough to consider the singular subvarieties of an ambient space.
{For an equivariant vector bundle $E$ over $X$ let
$$\lambda_{y}(E)=\sum_{k=0}^{{\rm rk}(E)}(-1)^k\Lambda^k E \in \K(X)\,.$$}
If $X$ is smooth, then by $\mC(X)$ we mean
\begin{equation}\label{eq:mC-norm}\mC(id_X)=\lambda_y(T^*X)=\sum_{k=0}^{\dim X} [\Omega^k_X] y^k\in \K(X)\,,\end{equation}
where $\Omega^k_X$ stands for the sheaf of $k$-differential forms on $X$.
If $X\subset M$ is a smooth but not necessarily closed subvariety, then we can compute $\mC(X\subset M)$ using functorial properties of the motivic Chern class. Let $f:Y\to M$ be a proper map which resolves the singularities of the closure of $X$. Suppose $f_{|f^{-1}(X)}:f^{-1}(X)\to X$ is an isomorphism and the complement of $f^{-1}(X)$ is a simple normal crossing divisor $D=\bigcup_{i\in I} D_i$. For $J\subset I$ set $D_J=\bigcup_{i\in J} D_i$, $D_\emptyset=Y$. Then
\begin{equation}\label{def:mc}\mC(X\subset M)=\sum_{J\subset I}(-1)^{|J|}(f_{|D_J})_*\big(\mC(D_J)\big)\,.\end{equation}
In terms of the sheaf of logarithmic forms on $Y$
\begin{equation*}\mC(X\subset M)=f_*\big(\O_Y(-D)\otimes\lambda_y\Omega^1_Y(\log D)\big)\,,\end{equation*}
see \cite[p. 19]{MS}.
The above formula is explicit, but to apply it one has to know a good resolution of the closure of $X$. It is not always easy to find one. The equivariant situation is somehow easier. In \cite{FRW, FRW2} there was given a characterization of the motivic Chern classes in terms of support and the Newton polytope, provided that an algebraic group acts on $M$ with finitely many orbits and we compute the motivic Chern class of an orbit.

\section{Definition of the twisted motivic Chern class}
\label{sec:def}
Let $G$ be an algebraic linear group, possibly trivial. Later on we will assume that $G$ is a torus, but for the definition of the twisted motivic Chern classes it is irrelevant. The definition given below is meaningful also for the trivial group.
\medskip

Let $X$ be an algebraic variety with a distinguished closed subset $\partial X$, which we call {\it the boundary}. We assume that $X^o=X\setminus \partial X$ is smooth. Let $\Delta$ be a $\Q$-Cartier divisor on $X$ with support contained in the boundary $|\Delta|\subset \partial X$.
{
\begin{definition}[cf. {\cite[Definition 9.1.2]{La}}]
	Let $\Delta=\sum q_iD_i$ be a $\Q$-divisor on $X$. We define the round-up divisor $\lceil D\rceil$ as
	$$\lceil D\rceil=\sum\lceil q_i\rceil D_i \,.$$
\end{definition}
}
\begin{definition}\label{def-twisted}\rm The twisted motivic Chern class of the triple $(X,\partial X; \Delta)$ as above is defined by the formula
$$\mC(X,\partial X;\Delta)=f_*\Big(\O_Y(\lceil f^*(\Delta)\rceil) \cdot \mC(Y^o\subset Y)\Big) \,,$$
where  $f:(Y,\partial Y)\to (X,\partial X)$ is a resolution of singularities such that $\partial Y=f^{-1}(\partial X)$ is a simple normal crossing (SNC for short) divisor, $Y^o=Y\setminus \partial Y$.

We assume that $f_{|Y^o}:Y^o\to X^o$ is an isomorphism.
Later we will skip $\partial X$ in the notation, when $\partial X=|\Delta|$.  \end{definition}
The twisted motivic Chern class is an element of the K-theory of (equivariant) coherent sheaves tensored with the polynomial ring $\Z[y]$. In our application we will study the image of $\mC(X,\partial X;\Delta)$ in the K-theory of a smooth ambient space. There the K-theory of coherent sheaves is isomorphic to the K-theory of locally free sheaves.

Before proving that the definition does not depend on the choice of the resolution let us show the simplest example.

\begin{example}[Basic calculation]\rm
Let $X=\P^1$, $\partial X=\{0\}$ with the standard {$\C^*=\T$} action. The tangent character at $0\in (\P^1)^\T$ is equal to $t$ and the character at $\infty$ is equal to $t^{-1}$.  Let $\Delta=\lambda\{0\}$, $\lambda\in\Q$.
Since $\P^1$ is already smooth, in the Definition \ref{def-twisted} we take $\P^1=Y$ and $f=id$, hence
$$\mC(\P^1;\Delta)=\O(\lceil \Delta\rceil)\cdot \mC(X\setminus |\Delta|\subset X)=\O(\lceil \lambda\rceil \{0\})\cdot\mC(\P^1\setminus\{0\}\subset \P^1)\,.$$
The weights of the divisor $\Delta$ at the fixed points are the following
$$\we_0(\Delta)=\lambda\,,\qquad \we_\infty(\Delta)=0$$
and the restrictions of $\mC(X;\Delta)$  are the following
$$ \mC(\P^1;\Delta)_{|0}=t^{\lceil \lambda\rceil}\cdot(1+y)t^{-1}=(1+y)t^{\lceil \lambda\rceil-1}\,,\qquad
\mC(\P^1;\Delta)_{|\infty}=1+yt\,.$$
We will consider the Newton inclusion property, see Definition \ref{def:enve}(3)
\begin{equation}\label{newton-intro}\Ne^{ \T}(\mC(X;\Delta)_{|0})-\we_0(\Delta)\subset \Ne^{ \T}(eu(TX)_{|0})\,.\end{equation}
The Newton polytopes in this example are the following
$$\Ne^{ \T}(\mC(X;\Delta)_{|0})=\{\lceil \lambda\rceil-1 \}\,,$$
$$\Ne^{ \T}(eu(TX)_{|0})=[-1,0].$$
The first one consists of a single element, the second one is an interval in $\Hom(\T,\C^*)\otimes\Q\simeq \Q$.
The Newton polytope inclusion clearly holds
$$\{\lceil \lambda\rceil-1 \}-\lambda=\{\lceil \lambda\rceil-\lambda-1 \}\subset [-1,0]\,.
$$
For  $\lambda\not\in \Z$ the first polytope is contained in the interior of the second.
\end{example}
{
\begin{remark}\rm
	Suppose that $\Delta_1$ is an integral Cartier divisor. Then for an arbitrary \hbox{$\Q$-Cartier} divisor $\Delta_2$
	$$\mC(X,\partial X;\Delta_1+\Delta_2)=\O_X(\Delta_1)\mC(X,\partial X;\Delta_2) $$
	by projection formula. This agrees with the behaviour of stable envelopes for  $G/B$,  \cite[Lemma 8.2c]{AMSS}. See also Remark \ref{rk:funktorialnosc} below.
\end{remark}
\begin{proposition} \label{pro:deform}
	Let $\Delta_1$ be an arbitrary \hbox{$\Q$-Cartier} divisor.  Let $\Delta_2$ be an effective $\Q$-Cartier divisor. Suppose that  $\Delta_2$ is small enough. Then
	$$\mC(X,\partial X;\Delta_1-\Delta_2)=\mC(X,\partial X;\Delta_1)\,. $$
	For example
	$$\mC(X,\partial X;-\Delta_2)=\mC(X^o\subset X)\,. $$
\end{proposition}
\begin{proof}
	The resolution of singularities $f$ is surjective and birational, therefore the divisor $f^*\Delta_2$ is effective. If the coefficients of $\Delta_2$ are small enough then
	$$\lceil f^*\Delta_1-f^*\Delta_2\rceil=
	\lceil f^*\Delta_1\rceil $$
	due to semi-continuity of the ceiling function.
\end{proof}
}

\begin{remark}[About relation with multiplier ideals, {\cite{BlLa, La}}]\rm
The construction of the twisted motivic Chern class has a lot in common with the definition of multiplier ideals \cite[Def 9.2.1]{La}. {Let $X$ be a smooth variety of dimension $n$, $\Delta$  an effective $\Q$-divisor and $\partial X$ any subvariety such that
$|K_X+ \Delta|\subset \partial X$. Then the coefficient of $y^n$ of the class
\begin{align*}
&(-1)^n\mC(X,\partial X;-K_X-\Delta)
\end{align*}
is equal to the K-theory class of the multiplier ideal for $\Delta$.} \end{remark}

In our consideration the singular variety $X$ plays no role, and we can work with a system of compatible smooth varieties with SNC $\Q$-divisors. Instead of a Cartier divisor $\Delta$ we may consider a fractional power of an ideal, as in the definition of the multiplier ideal associated with a formal power of an ideal, see  \cite[Def 9.2.3]{La}.
We will study these classes elsewhere.

\section{Motivation: Twisted motivic Chern class as a limit of elliptic classes}
\subsection{The elliptic class}
The elliptic class of a singular pair was defined by Borisov and Libgober \cite{BoLi0}. The torus equivariant version, when there is only a finite number of fixed points was considered in \cite[\S2.4, formula 8]{RW}, \cite{MW}. In that case, and additionally when $X$ admits a resolution with finitely many fixed points, the elliptic class can be computed by quite a compact formula. The formulas below define a class in the equivariant K-theory (extended by formal parameters $q$ and $h$). Applying the Riemann-Roch transformation we obtain an element in equivariant cohomology, which was studied in \cite{BoLi0,BoLi}.
\medskip

First let us recall the basic definitions. We follow \cite{RW, MW}.
Let $q=e^{2\pi i\tau}$ belong to the unit disk in $\C$, or consider $q$  as a formal parameter.
The multiplicative version of the Jacobi theta function is defined by the convergent (or formal, if $q$ is treated as a formal variable) infinite product
$$\vt(x)=(x^{1/2}-x^{-1/2})\prod_{n={1}}^\infty (1-q^nx)(1-q^n/x)\,.$$
Let $$\delta(x,y)=\frac{\vt'(1)\vt(xy)}{\vt(x)\vt(y)}=\frac{1-(xy)^{-1}}{(1-x^{-1})(1-y^{-1})}+\sum_{n=1}^\infty q^n\sum_{k|n}(x^{-k}y^{-n/k}-x^ky^{n/k})$$
$$=\frac{1-(xy)^{-1}}{(1-x^{-1})(1-y^{-1})}+q(x^{-1}y^{-1}-xy)+\mathcal O(q^2)$$
be the
function defined in \cite[Sec.~3, p.~456]{Za91}, \cite{RTV}, \cite[\S2.1]{RW}. The following properties of $\delta$ are fundamental:
\begin{equation}
\delta(x,y)=\delta(y,x),\qquad \delta(x^{-1},y^{-1})=-\delta(x,y),
\end{equation}
\begin{equation}\label{period}
\delta(x,q^{-1} y)=x\,\delta(x,y)\,.
\end{equation}

The elliptic class of the pair $(X,\Delta)$ is equal to the push-forward of the elliptic class of $(Y,D)$, where $f:Y\to X$ is a resolution of $X$ and the divisor $D$ satisfies $$K_Y+D=f^*(K_X+\Delta)\,.$$ To make sense of the pull-back we have to assume that  $K_X+\Delta$ is $\Q$-Cartier. For a smooth variety with a SNC divisor the original definition of \cite{BoLi0} is quite complicated, but if we assume that $\T$ has a finite number of fixed points on $Y$, then the following formula is efficient. Suppose that $p\in Y$ is a fixed point and that the divisor $D$ is given by sum of the coordinate hyperplanes in some local coordinates. Assume that $\T$ acts with the weight $w_k$ on $k$-th coordinate and
assume that the multiplicity of $D$ in that direction is equal $a_k$. {It follows from \cite[section 2.3]{RW} that}
\begin{equation}\label{def-ell}\Ell(Y,D)_{|p}=eu(T_pY)\prod_{k=1}^{\dim Y}\delta(t^{w_k},h^{1-a_k})\,.\end{equation}
Here
$$eu(T_pY)=\prod_{k=1}^{\dim Y}1-t^{-w_k}$$
is the K-theoretic Euler class, $h$ is a formal variable and we allow fractional powers of $h$. The expression above makes sense when $a_k\neq 1$, otherwise we have a zero in the denominator of $\delta$, since $\vt(h^0)=\vt(1)=0$. To show that the definition does not depend on the resolution we have to assume that $a_k<1$. In other words we assume that the pair $(X,\Delta)$ has at worst Kawamata log-terminal singularities.

\subsection{Limit properties of $\delta$}
We are interested in the limits of the elliptic class when $q\to 0$, or equivalently $\tau \to i\infty$. We have an obvious way of passing to the limit
\begin{equation}
\lim_{q\to0} \delta(x,y)=\frac{1-(xy)^{-1}}{(1-x^{-1})(1-y^{-1})},
\end{equation}
Then according to \cite[Prop 3.9]{BoLi0}
$$\lim_{\tau \to i\infty}\int_X\Ell(X,D)=c\, E_{st}(X,D)[h^{-1},1]\,,$$
where $E_{st}(X,D)\in \Z[u,v]$ is the stringy Hodge polynomial defined by Batyrev, the normalizing factor is equal  to $$c=(h^{1/2}-h^{-1/2})^{-\dim X}.$$
We propose to call the limit $\lim_{\tau \to i\infty}\Ell(X,D)$  the stringy motivic Chern class and denote it by $\mC_{st}(X,D)$. Applying the Riemann-Roch transformation we obtain the stringy Hirzebruch class mentioned in \cite[formula 3.9]{BSY} or \cite[\S11.4]{SY}.
 In the case of isolated fixed points the stringy motivic Chern class is obtained  as $\Ell(X,D)$, but instead of $\delta(t,h^{1-a_k})$ in the formula \eqref{def-ell} we use the factors
$$\lim_{q\to 0}\delta(t^{w_k},h^{1-a_k})= \frac{1-t^{-w_k}h^{a_k-1}}{(1-t^{-w_k})(1-h^{a_k-1})}\,.$$
We will not study these classes in the present paper.
\subsection{A smarter way of passing to the limit}

By elementary analysis argument and using the periodicity property of $\delta$ we have
\begin{equation}
\lim_{\tau \to i\infty} \delta(x,q^{-\lambda} y)=\begin{cases}\frac{y^{\lambda}(1-(xy)^{-1})}{(1-x^{-1})(1-y^{-1})}&\text{ if }\lambda\in \Z\\ \\
\frac{x^{\lfloor \lambda\rfloor}}{1-x^{-1}}&\text{ if }{\lambda \not\in \Z.}\end{cases}
\end{equation}
see \cite[proof of Prop. 3.13]{BoLi}, \cite[\S2.1]{RTV}.
Here to make sense of $q^{-\lambda}$ we set $q^{-\lambda}:=e^{-2\pi i \tau \lambda}$, and instead $q\to 0$ we take $\tau\to i\infty$.
\medskip

The elliptic class is  defined when the coefficients $a_k$ are smaller than one.
Let us consider the coefficients of the form $a_k=1-\varepsilon d_k$. Then the factors in the expression \eqref{def-ell} are equal to
$$\delta(t^{w_k},h^{1-a_k})=\delta(t^{w_k},h^{\varepsilon d_k})\,.$$
Let us formally substitute $h^{-\varepsilon}$ by $q$. Then we obtain the factors
$$\delta(t^{w_k},q^{-d_k})$$ and (assuming that the coefficients $d_k$ are not integral)
\begin{equation}
\lim_{\tau \to i\infty}\delta(t^{w_k},h^{1-a_k})=
\lim_{\tau \to i\infty} \delta(t^{w_k},q^{-d_k} )=
\frac{t^{\lfloor d_k\rfloor w_k}}{1-t^{-w_k}}
\,.
\end{equation}
The localized motivic Chern class is of the product form
$$\mC(Y\setminus |D|\to Y)_{|{p}}={eu(T_{p}Y)}\cdot\prod_{k=1}^m (y+1)t^{-w_k}\cdot \prod_{k=m+1}^{\dim Y}( 1+yt^{-w_k})\,,$$
provided that the divisor $D=\sum_{k=1}^m d_k D_k$, $d_k\neq 0$ and $c_1(\O(D_k))_{|p}=t^{w_k}$.
The localized twisted class (see Definition \ref{def-twisted}) is of the form
$$\mC(Y;D)_{|{p}}={eu(T_{p}Y)}\cdot\prod_{k=1}^m t^{\lceil d_k\rceil w_k}\cdot\prod_{k=1}^m (y+1)t^{-w_k}\cdot \prod_{k=m+1}^{\dim Y}( 1+yt^{-w_k})\,.$$
Since for non-integral coefficients $d_k$ we have $\lceil d_k\rceil-1=\lfloor d_k\rfloor$ the class
$\mC(Y;D)$ coincides with the refined limit of the elliptic class multiplied by $(1+y)^{\dim Y}$. Hence
\begin{equation*}(1+y)^{\dim Y}\lim_{q\to0}\Ell\left(Y,\sum\nolimits_k(1-\varepsilon d_k) D_k\right)_{h^{-\varepsilon}:=q}=\mC\left(Y;\sum\nolimits_k d_k D_k\right)\,.\end{equation*}
\medskip
Suppose $(X,\Delta_0)$ is a pair admitting a resolution $f:Y\to X$ such that $D=f^*(K_X+\Delta_0)-K_Y$ has the coefficients equal to one. Then the limit
\begin{equation*}\lim_{q\to0}\Ell\left(X,\Delta_0+\varepsilon \Delta\right)_{h^{-\varepsilon}:=q}\end{equation*}
exists and
\begin{equation}\label{ell2cm}(1+y)^{\dim X}\lim_{q\to0}\Ell\left(X,\Delta_0+\varepsilon \Delta\right)_{h^{-\varepsilon}:=q}=\mC\left(X;\Delta\right)\,.\end{equation}
Such a situation happens for Schubert varieties in the homogeneous spaces, see \cite{RW,KRW}.\medskip

Note that the elliptic class is defined only for a special type of pairs (Kawamata log-terminal pairs), while for the twisted motivic Chern class there is no such restriction. We demand only that the divisor $\Delta$ is $\Q$-Cartier.
\medskip

We conclude with the remark, that the formula \eqref{ell2cm} remains true for not necessarily isolated fixed points. Only the argument (and the definition) would be considerably more involving.

\section{Independence from the resolution}
\label{sec:niez}
To prove that the class given in the Definition \ref{def-twisted} does not depend on the resolution we apply the Weak Factorization Theorem \cite{AKMW}, \cite[Th.~0.0.1]{Wlodek}. It is enough to consider the single blow-up in a center $C\subset \partial Y$ which has normal intersections with components of the divisor $\partial Y=\bigcup_{k=1}^m D_k$.
Let $Z$ be a blow-up of $Y$ in such center:
$$Z=Bl_CY\stackrel{b}\longrightarrow Y\stackrel{f}\longrightarrow X\,.$$
 Suppose $C$ is contained in the intersection of the boundary components $\bigcap_{k=1}^r D_k$ and is not contained in any divisor $D_i$ for $i>r$. Let $\partial Z=b^{-1}(\partial Y)$, $Z^o=Z\setminus\partial Z$ and $D=f^*(\Delta)$.
Suppose $D=\sum_{k=1}^m c_k D_k$.
The multiplicity of $E$ in $b^*(D)$ is equal to $\sum_{k=1}^r c_k$.
The difference{$$b^*(\lceil D\rceil)-\lceil b^*(D)\rceil=\left(\sum_{k=1}^r \lceil c_k\rceil-\left\lceil\sum_{k=1}^r c_k\right\rceil\right)E=sE\,,$$}where $E$ is the exceptional divisor of the blowup $b$ and $s$ is a nonnegative integer smaller than $r$.
We need to show  that the push-forward
$$b_*\Big(\O_{Z}(\lceil b^*(D)\rceil) \cdot \mC(Z^o\subset Z)\Big)$$
is equal to
$$\O_{Y}(\lceil D\rceil) \cdot \mC(Y^o\subset Y).
$$
By the projection formula it is enough to prove that
$$b_*\Big(\O_{Z}({-s}E) \cdot \mC(Z^o\subset Z)\Big)
=
\mC(Y^o\subset Y)$$
for $s\in[0,r-1]$.

\subsection{Comparison of twists}
Our calculus is based on the following formula:
Let $Z$ be a smooth variety and $E$ a smooth divisor. Let
$I_E=\O_Z(-E)\subset \O_Z$ be the ideal sheaf of $E$. Let $i:E\to Z$ denote the inclusion.
The exactness of the sequence on $Z$
$$0\to \O_Z(-E)\to \O_Z\to i_*(\O_E)\to 0$$
implies the equality
\begin{equation}\label{roznica}\alpha- \O_Z(-E)\cdot\alpha=i_*(\O_E) \cdot \alpha \in \K(Z)\end{equation}
for any $\alpha\in \K(Z)$.

\subsection{Adjunction exact sequence}
Let $L= i^*\O_Z(-E)$.
The following sequence of sheaves on $E$ is exact
$$0\to\Omega^{p-1}_E\otimes L \to i^*\Omega^p_Z\to \Omega^p_E\to 0\,.$$
Here the second map is induced by the differential
$$d:I_E\to \Omega^1_Z\,,$$
and the third is the restriction of forms.
Therefore in the K-theory we have an equality
$$i^* \Omega^p_Z=\Omega^p_E+ L\cdot {\Omega^{p-1}_E}\,,$$
or
$$i^* \mC(Z)=(1+y\,L)\cdot \mC(E)\,.$$
{By the proper functoriality of $\mC$ classes we have}
\begin{equation}\label{mc_hiperpowierzchni}i^*\mC(E\subset Z)=i^*i_*\mC(E)=(1-L)\cdot \mC(E)\,.\end{equation}
{Moreover, by the additivity of $\mC$ classes we have}
\begin{align}\notag i^*\mC(Z\setminus E\subset Z)&={\mC(Z)-\mC(E\subset Z)}\\
\notag &=(1+yL)\cdot \mC(E)-(1-L)\cdot \mC(E)\\
&=(1+y)\,L\cdot \mC(E)\label{obciecie}
\,.\end{align}

\subsection{Blow-up of a SNC divisor}
Let $Y$ be a smooth variety and let $\partial Y=\bigcup_{k=1}^m {D_k}$ be a SNC divisor. Suppose $C\subset \bigcap_{k=1}^r {D_k}$ for $r\leq m$ and $C\cap \partial Y_k=\emptyset$ for $k>r$. Let $b:Z\to Y$ be the blow-up of $C$ and let $E$ be the exceptional divisor. Let $i:E\hookrightarrow Z$ and $i':C\hookrightarrow Y$ be the inclusions.
Denote by $\partial Z_k$ the proper transform of ${D_k}$ and set $\D=\bigcup_{k=1}^m \partial Z_k$.
Let $$ {Y^o=Y\setminus\partial Y\,,}\qquad\partial Z=b^{-1}(\partial Y)=\D\cup E\,,\qquad Z^o=Z\setminus\partial Z=b^{-1}Y^o\,. $$ By the functorial property of the motivic Chern class
$$\mC(Y^o\subset Y)=b_*\left(\mC(Z^o\subset Z)\right)\in \K(Y)\,.$$
In our situation the line bundle $L=i^*\O_Z(-E)$ is the relative $\O_{E/C}(1)$ for the projective  bundle $b_{|E}:E\to C$.

\begin{lemma}\label{znik}
With the above notation
$$b_*\left(\O_Z(-sE)\cdot\mC(Z^o\subset Z) \right)=\mC(Y^o\subset Y)$$
for $0\leq s\leq r-1$.
\end{lemma}

\begin{proof}

We apply the formula \eqref{roznica}  for $\alpha=\O_Z((-s+1)E)\cdot\mC(Z^o\subset Z) $
and see that the difference between
\begin{equation}\O_Z((-s+1)E)\cdot\mC(Z^o\subset Z)\label{mcroznica1}
\end{equation}
and
\begin{equation}\O_Z(-sE)\cdot\mC(Z^o\subset Z)
\label{mcroznica2}
\end{equation}
is equal to
\begin{align*}
i_*\O_E\cdot \O_Z((-s+1)E)\cdot\mC(Z^o\subset Z)&=
 i_*i^*\left(\O_Z((-s+1)E)\cdot\mC(Z^o\subset Z)\right)\\
&=i_*\left(L^{s-1}\cdot i^*\mC(Z^o\subset Z)\right)\\
&=i_*\left((1+y)L^{s}\cdot \mC(E\setminus \D\subset E)\right)\,.
\end{align*}
The last equality follows from \eqref{obciecie} and the inclusion-exclusion formula for SNC divisors. Let us give names to the maps appearing in the blow-up diagram
$$\begin{matrix}E&\stackrel{i}\longrightarrow&Z\\
^g\downarrow&&\downarrow^b\\
C&\stackrel{i'}\longrightarrow&Y
\end{matrix}$$
The push forward $b_*$ of the difference \eqref{mcroznica1}--\eqref{mcroznica2} is equal to
$${(1+y)\cdot}i'_*g_*\left(L^{s}\cdot \mC(E\setminus \D\subset E)\right)$$
The conclusion follows from the lemma below.\end{proof}

\begin{lemma} Let $V\to C$ be a vector bundle and $\{V_k\}_{k=1}^r$ a collection of codimension one subbundles $V_k\subset V$. We assume that {$r\le \rank V$ and} the subbundles $V_k$ are in general position at each point of $C$. Consider the projective bundle $g\colon E=\P(V)\to C$ and the divisors $B_k=\P(V_k)$, $B=\bigcup_{k=1}^r B_k$. Then
$$g_*\left(\O_{E/C}(s)\cdot \mC(E\setminus B\subset E)\right)=0\qquad \text{for }s=1,2,\dots, r-1\,.$$
\end{lemma}
\begin{proof} Let $\Omega^1_{E/C}$ be the dual of the tangent bundle of the fibers of $g$. We apply the Euler sequence
$$0\to\Omega^1_{E/C}\to g^*V^*(-1)\to\O_E\to 0\,.$$
Here the twist $(-1)$ denotes the tensor product with the relative $\O_{E/C}(-1)$ bundle, for example $g^*V^*(-1)=g^*V^*\otimes \O_{E/C}(-1)$.
The motivic Chern class of $E$ can be presented as
$$\mC(E)=g^*(\mC(C))\cdot \lambda_y(\Omega^1_{E/C})=g^*(\mC(C))\cdot \lambda_y(g^*V^*(-1))/(1+y)\,.$$
{See Verdier-Riemann-Roch formula \cite[Theorem 4.2 (4)]{AMSS} for the first equality}. Similarly
$$\mC(B_k)=g^*(\mC(C))\cdot \lambda_y(g^*V_k^*(-1))/(1+y)$$
and
$$\mC(B_k\subset E)=(1-L_k(-1))\cdot g^*(\mC(C))\cdot \lambda_y(g^*V_k^*(-1))/(1+y)$$
since the ideal of $B_k$ in $E$, is isomorphic to $L_k^{-1}(1)$.
The factor $\lambda_y(g^*V^*(-1))$ decomposes
$$\lambda_y(g^*V^*(-1))=\prod_{k=1}^r (1+y L_k(-1))\cdot \lambda_y(K(-1))\,,$$
where $L_k=g^*ker(V^*\to V^*_k)$ and $K=g^*(V^*/\sum_{k=1}^r L_k)$.
{Divisors $B_k$ are SNC, therefore we may apply} the above decomposition to the projective bundles $B_I=\P(\bigcap_{k\in I} V_k)$ for any {multi-index} $I\in\{1,2,\dots, r\}$. By additivity of the motivic Chern classes we obtain
$$\mC(E\setminus B\subset E)=  g^*(\mC(C))\cdot (1+y)^{r-1}\prod_{k=1}^r L_k(-1)\cdot \lambda_y(K(-1))\,.$$
We will show that
$$R^pg_*\left(\O_{E/C}(s)\otimes\prod_{k=1}^r  L_k(-1)\otimes \lambda_y(K(-1))\right)=0\quad\text{ for }0<s<r\,,\;p\in \Z$$
already on the level of sheaves. The statement is local on $C$, hence we can assume that the bundles $L_k$ and $K$ are trivial.
Then
$$R^pg_*\left(\O_{E/C}(s)\otimes  \O_{E/C}(-r)\otimes (1+y\O_{E/C}(-1))^{\rk(V)-r}\right)$$
is a direct sum of the sheaves
$$R^pg_*\left(\O_{E/C}(s-r')\right)\quad \text{with } r\leq r'\leq \rk(V)\,.
$$
The statement follows from the vanishing {(cf. \cite[Exercise 8.4 (a)]{Hartshorne})}
$$R^pg_*\left(\O_{E/C}(-j)\right)=0\quad \text{for } 0<j<\rk(V)\,,\;p\in\Z\,.$$
\end{proof}
\begin{example}[Evidence by calculus]\rm
	Let us examine the extreme case: $$C=\bigcap_{k=1}^r D_k=pt\subset \C^r\,.$$ We consider the natural action of the torus $\T=(\C^*)^r$  and perform the local calculation. Then {$E=\P^{r-1}$ and $\D\cap E$ is a sum of coordinate hyperspaces. We want to compute
	$$g_*\left(\mC(E\setminus \D\subset E)\cdot L^s\right)=g_*\left(\mC(\P^{r-1}\setminus \D\subset \P^{r-1})\cdot \O_{\P^{r-1}}(s)\right)\,.$$
	By Lefschetz-Riemann-Roch theorem it is enough to compute restrictions of the above class to the fixed point set $\P^\T$. It can be done using the fundamental calculation of \cite[section 2.7]{FRW} and the product property of $\mC$ classes \cite[Theorem 4.2 (3)]{AMSS}. We obtain that the above pushforward is equal to}
	$$\sum_{k=1}^rt_k^{-s}\prod_{\ell\neq k}\frac{{(1+y)}t_k/t_\ell}{1-t_k/t_\ell}
	=(1+y)^{{r-1}}\prod_{\ell=1}^rt_\ell^{-1}\cdot\sum_{k=1}^r\frac{ t_k^{r-s}}{\prod_{\ell\neq k}(1-t_k/t_\ell)}
	$$
	{Due to Lefschetz-Riemann-Roch formula \cite[Thm.~3.5]{Tho} or \cite[Thm.~5.11.7]{ChGi} we have
		$$\sum_{k=1}^r\frac{ t_k^{r-s}}{\prod_{\ell\neq k}(1-t_k/t_\ell)}=\chi(\P^{r-1},\O_{\P^{r-1}}(s-r))$$
		which is zero for $s\in\{1,2,\dots,r-1\}$.}
\end{example}

The working assumption  of Lemma \ref{znik} that $C$ lies entirely in the intersection of components of $D$ may be removed.
\begin{theorem}\label{znik2} Suppose that $Y$ is a smooth variety and $\partial Y=\bigcup_{i\in I}D_i\subset Y$ is a simple normal crossing divisor. Let $C\subset Y$ be a smooth subvariety which intersects $\partial Y$ normally (i.e.~in a local coordinates the components of $\partial Y$ and the subvariety $C$ are given by vanishing of some variables). Suppose $C\subset \bigcap_{k=1}^r D_{i_k}$ for some pairwise distinct components $D_{i_k}\subset \partial Y$.
Denote by
 $b:Z\to Y$ the blowup of $Y$ in $C$. Then
$$b_*\left(\O_Z(-sE)\cdot\mC(Z^o\subset Z) \right)=\mC(Y^o\subset Y)$$
for $0\leq s\leq r-1$, where $Y^o= Y\setminus \partial Y$ and $Z^o=b^{-1}(Y^o)$.
\end{theorem}

\begin{proof} The proof is an exercise in the inclusion-exclusion formula.
For simplicity let us consider the case when $\partial Y= \bigcup_{k=0}^{r}D_{k}$ and  the center of the blow-up is entirely contained in all but one component of the divisor $$C\subset \bigcap_{k=1}^rD_k\,.$$ The general case follows by induction.  Let $E\subset Z$ be the exceptional divisor and $\widetilde{D}_0$ the proper transform of $D_0$. Then
$$b'=b_{|\widetilde{D}_0}:\widetilde{D}_0\to D_0$$
is the blow-up of $D_0$ in the center
$$C'=C\cap D_0\,.$$
That is so because $D_0$ meets $C$ normally. 
The class of the exceptional divisor of $b'$ is the restriction of the class $E$, since $E$ meets $\widetilde D_0$ transversely.
The center of $b'$ is entirely contained in the intersection of $r$ divisors in $D_0$
$$C'\subset \bigcap_{k=1}^r(D_0\cap D_k)\,.$$
Let us introduce the following notation
$$
Y^\#=Y\setminus \bigcup_{k=1}^r D_k\,,\qquad
Y^o=Y\setminus \bigcup_{k=0}^r D_k\,,\qquad
\widetilde{D}_0^o=b^{-1}\left(D_0\setminus \bigcup_{k=1}^r D_k\right)\,.$$
We have disjoint unions
$$Y^\#=Y^o\cup D^o_0\,,\qquad Z^\#=Z^o\cup \widetilde{D}^o_0\,.$$
Denote by $i:D_0\to Y$ and $j:\widetilde D_0\to Z$ the inclusions.
By lemma \ref{znik} we have
$$b_*\big(\O_Z(-sE)\cdot\mC(Z^\#\subset Z)\big)=\mC(Y^\#\subset Y)$$
and
$$b'_*\big(j^*\O_Z(-sE)\cdot\mC(\widetilde D^o_0\subset \widetilde D_o)\big)=\mC(D_0^o\subset D_0)\,.$$
for $0\leq s\leq r-1$. Therefore by additivity of the motivic Chern class
\begin{align*}&b_*\big(\O_Z(-sE)\cdot\mC(Z^o\subset Z)\big)=\\
&=
b_*\big(\O_Z(-sE)\cdot\mC(Z^\#\subset Z)\big)
-b_*\big(\O_Z(-sE)\cdot j_*\mC(\widetilde D^o_0\subset \widetilde D_0)\big)\\
&=
b_*\big(\O_Z(-sE)\cdot\mC(Z^\#\subset Z)\big)-b_*j_*\big(j^*\O_Z(-sE)\cdot \mC(\widetilde D^o_0\subset \widetilde D_0)\big)\\
&=
b_*\big(\O_Z(-sE)\cdot\mC(Z^\#\subset Z)\big)-i_*b'_*\big(j^*\O_Z(-sE)\cdot \mC(\widetilde D^o_0\subset \widetilde D_0)\big)\\
&=
\mC(Y^\#\subset Y)-i_*\big(\mC( D^o_0\subset  D_0)\big)\\
&=\mC(Y^o\subset Y)
\end{align*}
\end{proof}

\begin{corollary}Let $\L$ be a line bundle on $Y$.
With the above notation
$$b_*\left(b^*\L(-sE)\cdot\mC({Z^o}\subset Z) \right)=\L\cdot\mC({Y^o}\subset Y)$$ for $0\leq s<r$.
\end{corollary}

\begin{proof} The corollary follows from the projection formula and Theorem \ref{znik2}.\end{proof}

\begin{corollary}\label{cor:niezal} The twisted motivic Chern class given by Definition \ref{def-twisted} does not depend on the resolution.\end{corollary}

\begin{proof} By \cite{AKMW}, \cite[Th.~0.0.1]{Wlodek} any two resolutions can be joined by a sequence of blow-ups or blow-downs with centers normally intersecting the components of the boundary divisor.
For a blow-up in a center $C$ contained in the $r$-fold intersection $\bigcap_{j=1}^r D_j$
$$Z\stackrel{b}{\longrightarrow} Y\stackrel{f}{\longrightarrow} X\,.$$
let $\L=\O_Y(\lceil f^*\Delta\rceil)$.
Then $\O_Z(\lceil (fb)^*\Delta\rceil)=b^*\L(-sE)$ with $0\leq s<r$ and we apply the previous corollary.
\end{proof}

We record the following local and product properties (compare \cite[Corollary 2.1]{BSY}.

\begin{proposition}\label{prop:prod} Let $(X,\partial X)$ be a pair considered in the Definition \ref{def-twisted} and $\Delta$ a $\Q$-Cartier divisor supported by $\partial X$. Then

\begin{enumerate}[(i)]
\item For an open set $U\subset X$
$$\mC(U,\partial X\cap U;\Delta\cap U)=\mC(X,\partial X;\Delta)_{|U}\,.$$
\item For a smooth variety $S$
$$\mC(X\times S,\partial X\times S;\Delta\times S)=\mC(S)\boxtimes\mC(X,\partial X\;\Delta)\,.$$
\end{enumerate}
\end{proposition}
\begin{proof} The proposition follows from Corollary \ref{cor:niezal}. To compute the twisted motivic Chern classes in (i) and (ii) we can use the resolution of $(X,\partial X)$ restricted to $U$ or multiplied by $S$.\end{proof}

\section{Bia{\l}ynicki-Birula decomposition}
Suppose that one dimensional torus $\T=\C^*$ is acting on a smooth and complete variety $M$. Let $F\subset M^\T$ be a component of the fixed point set.
According to \cite{BB} the set
\begin{equation*}M_F^+=\{ x\in M\;|\; \lim_{t\to 0} t\cdot x\in F\}\end{equation*}
is a locally closed smooth subvariety. It is called the Bia{\l}ynicki-Birula (BB for short) cell. In the papers \cite{OM, OS, O2} it is called the attracting set.
The variety $M$ is a disjoint union of the cells
\begin{equation}\label{BB-dc}M=\bigsqcup_{F\subset M^\T} M^+_F\,.\end{equation}
Similarly we can define the dual cells (minus-cells)
\begin{equation}\label{eq:minus}M_F^-=\{ x\in M\;|\; \lim_{t\to \infty} t\cdot x\in F\}\,.\end{equation}
The sets $M^+_F$ and $M^-_F$ intersect transversely along $F$. If $F$ is a point, then both are isomorphic to affine spaces, otherwise the limit map $M^\pm_F\to F$ is an affine fibration.
{
\begin{definition}[cf. {\cite[Definition on p. 36]{StratMorse}}]
	We call a decomposition of a variety $M=\bigsqcup M_i$ into locally closed subsets a stratification when the following condition is satisfied
	$$\overline{M_i}\cap M_j\neq \varnothing \Rightarrow M_j\subset \overline{M_i} \,.$$
\end{definition}
}
The decomposition \eqref{BB-dc} does not have to be a stratification. Even for simple examples of surfaces it may happen that $M^+_{F_1}\cap \overline {M^+_{F_2}}\neq \emptyset$ but $M^+_{F_1}\not\subset \overline {M^+_{F_2}}$. The good news is that the decomposition is filterable, provided that $M$ is projective, see \cite{BB2}.
\medskip

When a bigger torus $\T\simeq(\C^*)^n$ is acting then we choose a one dimensional subtorus $\sigma:\C^*\to\T$. For a generic choice of $\sigma$ the fixed point set remains the same $M^\T=M^\sigma$. The fixed point set can be bigger when $\sigma$ belongs to a certain sum of hyperplanes \hbox{$\Sigma \subset \Hom(\C^*,\T)\otimes \R$.} The connected components of  $(\Hom(\C^*,\T)\otimes\R)\setminus \Sigma$ are called {\it co-weight chambers}.
The BB-decomposition depends only on the choice of the co-weight  chamber. In what follows we will fix a co-weight chamber and we do not study the dependence on the chamber.

\medskip
We will assume that $M$ is projective and that the BB-decomposition forms a stratification.
In the present paper we assume that $M^\T$ is finite.
Still we need a stronger assumption. We will assume that the stratification by BB-cells satisfies the Whitney conditions. This is necessary for the support axiom of the Okounkov envelopes \ref{def:enve}(1) presented below.
The Whitney condition guarantees that  the sum of the conormal bundles to BB-cells of $M$ is a closed subvariety of the cotangent space $T^*M$.
We will discuss this issue in \S\ref{sec:support}. The main subject of our consideration is the Newton inclusion property \ref{def:enve}(3). To prove that the twisted motivic Chern class has that property
we do not have to assume that the BB-decomposition is a stratification.

\section{Stable envelopes for the cotangent bundle} \label{s:stable}
 Let $N$ be a symplectic variety equipped with an action of algebraic torus $\TT$. Let $\T\subset \TT$ be a subtorus which {preserves} symplectic form. For a given polarization $T^{1/2}\in  K_\TT(N)$,
co-weight chamber $\mathfrak{C}$ and slope $s\in Pic(N)\otimes \Q$ the K-theoretic stable {envelope} is a characteristic class
$$Stab_{\mathfrak{C},T^{1/2}}^s\in K_\TT(N^\T\times N).$$
(We do not explain here the role of polarization, because in our case of the cotangent bundle there is a canonical choice of the polarization.)
Initially stable envelope was defined only for varieties which admit proper map to an affine variety (see \cite[paragraph 3.1.1]{OM}) such as Nakajima quiver varieties. It turns out that this condition may be {weakened}, it is used only to show that the sum of BB-cells (or stable manifolds) in $N$ is closed \cite[lemma 3.2.7]{OM}. This weaker condition is used in the recent paper \cite[paragraph 1.2.3]{O3} {which introduces stable envelopes in even more general context}.

We consider the simplest example of a symplectic variety, the cotangent bundle. Let $M$ be a smooth, projective variety endowed with an action of a torus $\T$. Suppose that the fixed point set $M^\T$ is finite. Consider the cotangent variety $N=T^*M$ with the action of the torus $\TT=\T\times\C^*$, where the group $\C^*$ acts on fibers by scalar multiplication. The subtorus $\T\subset\TT$ preserves the canonical symplectic form on $N$. Fix a co-weight chamber $\mathfrak{C}$ and polarization $T^{1/2}=TM.$ Chamber $\mathfrak{C}$ determines BB-decompositions of varieties $M$ and $N$. To define stable envelopes in this situation we need to check that sum of the BB-cells in $N$ is a closed subvariety. Note that $N^\T=M^\T$ and for a fixed point $\ee\in M^\T$ we have the BB-cell
$$M^+_\ee\subset M$$
and the BB-cell in $T^*M$
$$N^+_\ee\subset T^*M\,,$$
which is equal to the conormal bundle of $M^+_\ee$.

\begin{lemma}
	The subset
	$$\bigsqcup_{\ee\in M^\T} N_\ee^+\subset T^*M $$
	is closed if and only if the decomposition
	$$M=\bigsqcup_{\ee\in M^\T} M_\ee^+$$
	is a stratification which satisfies Whitney A condition.
\end{lemma}
\begin{proof}
	Note that the subset $\bigsqcup_{\ee\in M^\T} N_\ee^+\subset T^*M $ is constructible. It follows that it is Zariski closed if and only if it is closed in Euclidean topology. The Whitney condition A, when translated into language of  conormal bundles means exactly  that the union of conormal bundles to strata is closed in Euclidean topology (cf. \cite[section 1.8 page 44]{StratMorse}, \cite[page 186]{SchIMPAN}).
\end{proof}
In the rest of this section we consider a smooth projective variety $M$ such that the BB-decomposition is a Whitney stratification.
 For a fixed point  $\ee \in M^\T$ the positive part of tangent space $T_\ee M$  is denoted by $T_\ee^+M$, the negative part by $T_\ee^-M$. Denote by $\hbar$ the character of $\TT$ which is the projection on $\C^*$ and let
\begin{equation}\label{eq:nu}\nu_\ee^-=T^-_\ee M\oplus (\C_\hbar\otimes(T^+_\ee M)^*)\end{equation} be the negative   part of the vector space $T_\ee(T^*M)=T_\ee M\oplus (\C_\hbar\otimes T^*_\ee M)$. It is equal to the normal space of the BB-cell $(T^*M)^+_\ee$ at $\ee$.
We have a partial order on $M^\T$ induced by the chosen co-weight chamber $\mathfrak{C}$. Note that since $M^\T$ is discreet by our standing assumption, $M$ is the sum of a finite number of BB-cells, hence $Pic(M)$ is a free abelian group. We remind equivalent  version of the stable envelope axioms  for the cotangent bundle.
 \begin{definition}\label{def:enve} \rm
Choose a {generic} slope, i.e. a {generic} fractional line bundle $$s=\L^{1/n}\in Pic(M)\otimes\Q\,.$$
The K-theoretic stable envelope on $T^*M$
 is a set of elements
  $$Stab^s(\ee)\in K_{\T\times\C^*}(M)\simeq K_\T(M)[\hbar,\hbar^{-1}]$$
  indexed by the fixed points such that:
  \begin{enumerate}
   \item {\it Support Axiom.} For any fixed point $\ee_0$
   $$Stab(\ee_0)\quad\text{is supported by}\quad\bigsqcup_{\ee \le \ee_0} (T^*M)^+_{\ee} \,.$$
   \item {\it Normalization.}  For any fixed point $\ee_0$
   $$Stab(\ee_0)_{|\ee_0}=
   (-1)^{\dim (T_{\ee_0}^+M)}\frac{eu(\nu^-_{\ee_0})}{\det T_{\ee_0}^+M} \,.$$
   \item {\it Newton inclusion property.} Choose any linearization of the slope $s$. For a pair of fixed points  $\ee_0,\ee$ such that $\ee_0>\ee$ we demand  containment of Newton polytopes
   $$\Ne^\T\left(Stab(\ee_0)_{|\ee}\right)+\we_{\ee_0}(s)
   \subseteq
   \Ne^\T\left(eu(\nu^-_{\ee})\right)-\we_\ee(\det T_\ee^+M) +\we_{\ee}(s).$$
  \end{enumerate}
 \end{definition}

The stable envelope for a generic slope is unique, see \cite[Prop. 9.2.2]{O2}. For a proof adapted to the case of isolated fixed points see \cite[Prop. 3.6]{K}.
{
\begin{remark} \label{rem:sgen}
	The above three axioms define a unique class only for generic slope. One may strengthen the Newton inclusion property axiom (see \cite[Appendix A]{K}). For generic slopes the strengthened axiom is equivalent to \eqref{def:enve}(3), while for non-generic slopes it determines a K-theory class.	
\end{remark}
}
 \begin{remark}\rm Above axioms agree with  \cite[section 2.1]{OS}, \cite[section 9.1]{O2}  up to normalization. Namely
$$Stab^s(\ee)=\hbar^{-\frac{1}{2}\dim M_\ee^+}Stab_{\mathfrak{C},TM}^s(\ee)$$
  as proven in \cite[{Appendix A}]{K}. We normalize the stable envelope by the factor $\hbar^{-\frac{1}{2}\dim M_\ee^+}$ to avoid fractional powers of $\hbar$.
 \end{remark}
\begin{remark} \rm
	The assumption that the set
	$\bigsqcup_{\ee\in M^\T} (T^*M)^+_{\ee}$
	is a closed subvariety of $T^*M$ is necessary to state the support axiom.
\end{remark}
\begin{lemma}\label{lem:Newton}
 To prove the Newton polytope axiom  it is enough to show that
 \begin{equation}\label{eq:newton}\Ne^{\T}\left(Stab(\ee_0)_{|\ee}\right)\subseteq \Ne^{\T}\left(eu(T_\ee M)\right)+\we_{\ee}(s)-\we_{\ee_0}(s).\end{equation}
\end{lemma}
\begin{proof}
 It follows from the proof of \cite[Proposition {5.1}]{K} that
 $${\Ne^\T\left(eu(T_\ee M)\right)\subseteq\Ne^\T\left(eu(\nu^-_{\ee})\right)-\we_\ee(\det T_\ee^+M) \,.}$$
\end{proof}

\begin{remark}\rm \label{rk:funktorialnosc}
The stable envelopes are functorial in the following  general sense. Suppose we have an  automorphism of the pair $(\T,M)$. Like in the case of the Weyl group action on $G/B$, the automorphism $\varphi:M\to M$ does not have to preserve the $\T$-action  but we assume that there is a group automorphism $\psi:\T\to\T$ such that for $t\in \T$ and $x\in M$ we have $\varphi(t\cdot x)=\psi(t)\cdot\varphi(x)$. Then $$\varphi^*(Stab_{\mathfrak{C},T^{1/2}}^s(\ee)))=Stab_{\psi^{-1}(\mathfrak{C}),T^{1/2}}^{\varphi^*(s)}(\varphi^{-1}(\ee))\,,$$  since the stable envelopes are defined intrinsically. (Of course $\varphi$ preserves the canonical polarization $T^{1/2}=TM\subset T(T^*M)$.)
For the case of ${M}=G/B$  the natural Weyl group action
and the resulting invariance of the stable envelopes is given by \cite[Lemma 8.2a]{AMSS}.
A detailed picture of the action of the Weyl group on the set of stable envelopes corresponding to various slopes is given in \cite{SZZwall}.
There is also another effect -- the Grothendieck duality, see \cite[Lemma 8.2b]{AMSS}, which involves the change of polarization to the opposite {one}. We will not discuss it here.
\end{remark}

\section{Twisted classes as stable envelopes}
\label{sec:setup}
 Let $M$ be a smooth projective variety with a torus $\T$ action and let us assume that $M^\T$ is finite.
We focus on the main application of twisted motivic Chern classes considered in this paper -- comparison with the stable envelopes. We will take as the singular space $X$ the closure  of a Bia{\l}ynicki-Birula cell.
Precisely: we  choose a one parameter subgroup $\C^*\subset \T$, such that $M^{\C^*}=M^\T$ and
we fix a fixed point $\ee_0 \in M^\T$. We set $X^o=M^+_{\ee_0}$. Let $X$ be the closure of $X^o$ in $M$, $\partial X=X\setminus X^o$. Denote by $i$ the inclusion $i:X\subset M$.
For a $\Q$-divisor $\Delta$ with the support contained in $\partial X$ we have defined an element in the K-theory -- the twisted motivic Chern class. We fix a slope, i.e. an element $s=\L^{1/n}\in Pic(M)\otimes \Q$, a formal root of a line bundle $\L$.
In section \ref{sec:dependency} we will show that the slope determines a $\Q$-Cartier divisor $\Delta_{\ee_0,s}$ in $X$ having the support contained in $\partial X$.  We aim to show that the classes
\begin{align*}
i_*\mC(X,\partial X,\Delta_{\ee_0,s}) \in K_\T(M)[y]
\end{align*}
and
\begin{align*}
Stab^s(\ee_0) \in K_\T(T^*M)[h,h^{-1}]
\end{align*}
agree up to normalization (where $Stab^s(\ee_0)$ denotes the stable envelope for the cotangent variety $T^*M$). Namely let $\pi$ be the projection $\pi: T^*M\to M$ and let $\rho$ be the map
$$\rho \colon K_\T(T^*M)[y] \to K_\T(T^*M)[h,h^{-1}] $$
given by $\rho(y)=-h$. Our goal is the following theorem
\begin{theorem}\label{tw:podsumowanie}
	The class
	$$h^{-\dim(M^+_{\ee_0})}\rho\left(\pi^*i_*\mC(X,\partial X;\Delta_{\ee_0,s})\right) $$
	satisfies the normalization axiom and the Newton inclusion property of the stable envelope $Stab^s(\ee_0)$. Moreover if the stratification is regular enough (see discussion in section \ref{sec:support}) this class satisfies the support axiom and {for a generic slope} we have an equality:
	$$h^{-\dim(M^+_{\ee_0})}\rho\left(\pi^*i_*\mC(X,\partial X;\Delta_{\ee_0,s})\right)=Stab^s(\ee_0). $$
\end{theorem}
{
	\begin{remark}
		In the above theorem we do not assume that the slope $s$ is generic. The twisted motivic Chern class may be defined for an arbitrary slope and coincides with the stable envelope for a general slope. Due to proposition \ref{pro:deform}
		the $\mC$ class for a slope $s$ which is not general enough is equal to the stable envelope for a small anti-ample deformation of $s$
	\end{remark}
}
\begin{corollary}
	Consider a homogeneous variety $G/P$ with the action of {the maximal} torus \hbox{$\T\subset P$.} For a fixed point $w\in (G/P)^\T$ there is an equality:
		$$h^{-\dim(G/P^+_{w})}\rho\left(\pi^*i_*\mC(\overline{G/P^+_{w}},\partial (G/P^+)_{w};\Delta_{w,s})\right)=Stab^s(w). $$
\end{corollary}

\begin{remark}
	Note that it is enough to check the support axiom for the class $$\rho\left(\pi^*i_*\mC(X,\partial X;\Delta_{\ee_0,s})\right).$$
	Moreover operations $\rho$, $\pi^*$ and multiplication by $h^{-\dim(M^+_{\ee_0})}$ do not change the Newton polytope. It follows that it is enough to check the Newton inclusion property for the class $$i_*\mC(X,\partial X;\Delta_{\ee_0,s}).$$
\end{remark}

\begin{lemma} \label{lem:pod}
	Consider a class $\alpha\in K_\T(M)[y]$. To prove that the class $\rho(\pi^*\alpha)$ satisfies support axiom of $Stab^s(\ee_0)$ it is enough to show that
	$$\alpha\quad\text{is supported by}\quad\bigsqcup_{\ee \le \ee_0} M^+_{\ee} $$
	and
	$$ \mC_{y}(M^+_{\ee})_{|\ee}=\lambda_{y}(T^*_{\ee}M^+_{\ee})\quad\text{divides}\quad\alpha_{|\ee}$$
	for every fixed point $\ee \le \ee_0.$
\end{lemma}
\begin{proof}
	It follows form \cite[Lemma 5.2-4]{RTV} or \cite[lemma {5.7}]{K}.
\end{proof}

\begin{remark}
	{It is possible to show that the twisted motivic Chern class satisfies also the stronger version of Newton inclusion property from \cite{K} (see remark \ref{rem:sgen}).}
\end{remark}

\section{Normalization axiom}
We consider situation described in section \ref{sec:setup}.
The normalization condition does not involve the slope and its verification is the same as in \cite[Prop. {5.1}]{K}. Let us repeat the calculations. The center of the cell $\ee_0\in M^+_{\ee_0}$ is a smooth point.
The twisted motivic Chern class localized at ${\ee_0}$ {equals} to $$\ii_*\mC(X,\partial X;\Delta_{\ee_0,s})_{\ee_0}=\lambda_{-1}((T_{\ee_0}^-M)^*)\cdot\lambda_y((T_{\ee_0}^+M)^*).$$
The first factor is present because of the inclusion map $\ii: M^+_{\ee_0}\to M$, the second factor by the normalization property of motivic Chern classes \eqref{eq:mC-norm}. {Note that $\ee_0$ lies in the interior $X^o$, while the divisor $\Delta_{\ee_0,s}$ is supported on the boundary $\partial X$. Therefore, the element $\O(\Delta_{\ee_0,s})_{|\ee_0}$ is trivial.}

We consider the extended torus action $\TT=\T\times \C^*$ on $T^*M$. Recall the formula \eqref{eq:nu} for the normal space to the conormal of the cell:
$$\nu^-_{\ee_0}=
T_{\ee_0} ^-M\oplus (\C_\hbar \otimes (T^+_{\ee_0} M)^*)$$
as the representation of $\TT$. Here $\C_\hbar$ denotes the natural representation of $\C^*$. Hence the K-theoretic Euler class at ${\ee_0}\in T^*M$ of $\nu^-_{\ee_0}$ is equal to
$$\lambda_{-1}((T_{\ee_0}^-M)^*)\cdot \lambda_{-1} (\C_{1/\hbar}\otimes T^+_{\ee_0} M)\,.$$
We apply the identity
$$\lambda_{-1}(V)=(-1)^{\dim(V)}\det(V)\cdot \lambda_{-1}(V^*)$$
for $V=\C_{1/\hbar}\otimes T^+_{\ee_0} M$ and obtain
\begin{align*}eu(\nu^-_{\ee_0})=&(-1)^{\dim(M^+_{\ee_0})}\det(\C_{1/h}\otimes T^+_{\ee_0} M)\cdot\lambda_{-1}((T_{\ee_0}^-M)^*)\cdot \lambda_{-1} (\C_\hbar\otimes (T^+_{\ee_0} M)^*)\\
=&(-1)^{\dim(M^+_{\ee_0})}h^{-\dim(M^+_{\ee_0})}\det(T^+_{\ee_0} M)\cdot\lambda_{-1}((T_{\ee_0}^-M)^*)\cdot \lambda_{-\hbar} ((T^+_{\ee_0} M)^*)
\,.\end{align*}
Setting $\rho(y)=-\hbar$ we obtain
$$(-1)^{\dim(M^+_{\ee_0})} \frac{eu(\nu^-_{\ee_0})}{\det(T^+_{\ee_0} M)}=
h^{-\dim(M^+_{\ee_0})}
\rho\left(\ii_*\mC(X,\partial X;\Delta_{{\ee_0},s})\right)_{\ee_0}\,,$$
which is exactly the normalization axiom Def.~\ref{def:enve}(2).

\section{Newton inclusion property}
We consider situation described in section \ref{sec:setup}.
We consider a slope $s=\L^{\frac1n}$. Suppose that $\Delta$ is an $\T$-equivariant $\Q$-Cartier divisor on $X$ such that $ |\Delta|\subset\partial X$, $n\Delta$ is Cartier, and $i^*\L\simeq \O_X(n\Delta)$.
{Due to \cite[Proposition 2.10]{Brion} any} linearization of $i^*\L$ differs from
the natural linearization (see \S\ref{notacja1}) of $\O_X(n\Delta)$ by a twist. For any fixed point $\ee\in X^\T$
$$\we_\ee(\L)-\we_\ee(\O_X(n\Delta))=\we_{\ee_0}(\L)-\we_{\ee_0}(\O_X(n\Delta))\,.$$
Note that $\we_{\ee_0}(\O_X(n\Delta))=0$. Therefore
$$\we_\ee(\L)-\we_{\ee_0}(\L)=\we_{\ee}(\O_X(n\Delta))\,.$$
\medskip

In this section we aim to show that the class
$i_*\mC(X,\partial X;\Delta)$
satisfies the Newton inclusion property of the stable envelope
$Stab
^{s}(\ee_0)$ (see section~\ref{s:stable}).
Namely, according to Lemma \ref{lem:Newton}, we need to prove that for a fixed point $\ee<\ee_0$ there is an inclusion:
\begin{align} \label{wyr:30}
\Ne^{ \T}\left(i_*\mC\left(X,\partial X;\Delta\right)_{|\ee}\right)\subset \Ne^{ \T}\left(eu\left(T_\ee M\right)\right)+\frac{\we_\ee(\L)-\we_{\ee_0}(\L)}{n}\,,
\end{align}
which is equivalent to
\begin{align} \label{wyr:3}
\Ne^{ \T}\left(i_*\mC\left(X,\partial X;\Delta\right)_{|\ee}\right)-\we_\ee(\Delta)
\subset \Ne^{ \T}\left(eu\left(T_\ee M\right)\right).
\end{align}
\bigskip

The rest of this section contains the proof of the main result:

\begin{theorem}\label{thm:main} The Newton inclusion property \eqref{wyr:3} holds.\end{theorem}

Consider a SNC resolution of singularities
$$f:(Y,\partial Y) \to (X,\partial X).$$
Set $D=f^*\Delta$.
For a fixed component $$F \subset Y^\T\cap f^{-1}(\ee)$$ denote by $f_F$ the restriction of $f$ to  $F$.
The Lefschetz-Riemann-Roch  formula (cf. \cite[Thm.~3.5]{Tho} or \cite[The 5.11.7]{ChGi} for quasi-projective varieties) implies that
$$
\frac{i_*\mC\left(X,\partial X;\Delta\right)_{|\ee}}{eu(T_\ee M)}=
\frac{{i_*f_*}\mC\left(Y,\partial Y;D\right)_{|\ee}}{eu(T_\ee M)}=
\sum_{F\subset f^{-1}(\ee) \cap Y^\T} f_{F*}\left(\frac{\mC\left(Y,\partial Y;D\right)_{|F}}{eu\left(\nu(F\subset Y)\right)}\right)\,.
$$
The Lefschetz-Riemann-Roch formula can be rewritten as
\begin{align}
\label{wyr:1}
\frac{t^{\we_\ee(-\Delta)}\,i_*\mC\left(X,\partial X;\Delta)\right)_{|\ee}}{eu(T_\ee M)}=
\sum_{F\subset f^{-1}(\ee) \cap Y^\T} f_{F*}\left(\frac{t^{\we_F(-D)}\,\mC\left(Y,\partial Y;D \right)_{|F}}{eu\left(\nu(F\subset Y)\right)}\right).
\end{align}
We will show that each summand of the above sum satisfies the Newton inclusion property
$$\Ne^{ \T}(t^{\we_F(-D)}\,\mC\left(Y,\partial Y;D \right)_{|F})\subset\Ne^{ \T}(eu\left(\nu(F\subset Y)\right))$$ and we will apply the following tool  to compare Newton polytopes.

\begin{lemma} [\cite{FRW}]\label{lem:FRW}
	Consider a smooth variety $Y$ with trivial action of a torus $\T$. Let $a$ and $b$ be two classes in the equivariant $K$-theory $K_\T(Y)$. Suppose that $b$ is invertible in the localized $K_\T(Y)[y]$. The following conditions are equivalent
	\begin{enumerate}
		\item There is an inclusion of Newton polytopes
		$$\Ne^{ \T}(a)\subseteq \Ne^{ \T}(b)$$
		\item For a general enough one parameter subgroup $\sigma:\C^*\to \T$
		the limit
		$$\lim_{\tt \to 0} \frac{a_{|\sigma}}{b_{|\sigma}} \in K(Y)[y]$$
		exists.
	\end{enumerate}
\end{lemma}
For the sake of brevity we will write $\lim_\sigma$ meaning the restriction to $\sigma$.
Here by the localized K-theory we mean $S^{-1}K_\T(Y)[y]$, where $S\subset K_\T(pt)$ is generated by elements of the form $1-L$, where $L$ is one dimensional representation of $\T$. The lemma for $Y=pt$ is given in \cite[Proposition 3.6]{FRW}. The generalization given here is straightforward.

\bigskip

The above lemma implies that to prove the inclusion (\ref{wyr:3})  it is enough to show that for a general enough one dimensional subtorus $\sigma:\C^*\to \T$ the limit of a single summand occurring in the equation (\ref{wyr:1})
exists. Using commutation of the limit map with push-forwards (for maps between spaces with the trivial action of $\T$)
we only need to prove the lemma:

\begin{lemma}\label{lem:limit-frac}Let $f:(Y,\partial Y)\to (X,\partial X)$ and $\Delta$  be as in Definition \ref{def-twisted}. Let   $Y^o=Y\setminus \partial Y$, $D=f^*\Delta$. Then the limit
	\begin{align*}
	\lim_\sigma\frac{t^{-\we_F(D)}\,\mC\left(Y,\partial Y;D\right)_{|F}}{eu\left(\nu(F\subset Y)\right)} \in K(F)[y]
	\end{align*}
	exists for a general one parameter subgroup.\end{lemma}


Below we state two lemmas concerning behaviour of limits.
\begin{lemma}\label{lem:prostowanie}
	Consider a  variety $F$ with the trivial action of $\T$. Let $a\in S^{-1}K_\T(F)$ and let  $\L\in Pic^\T(F)$ be an equivariant line bundle.  For any one parameter subtorus $\sigma:\C^*\to\T$ the limit
	$\lim_\sigma (\L\cdot a) $ exists if and only if the limit
	$\lim_\sigma (t^{\we_F(\L)}\cdot a)$ exists.
\end{lemma}
\begin{proof}
	The multiplication by the line bundle with trivial character
	$t^{-\we_F(\L)}\,\L$ commutes with the limit map.\end{proof}

\begin{lemma} \label{lem:conv}
	Consider a  variety $F$ with the trivial action of a torus $\T$ and an element of the localized $K$-theory $\alpha\in S^{-1}K_\T(F)[y]$.
	For a given  fractional weight $w$ suppose that there are characters $w_1,...,w_m$ such that
	$$w \in conv(w_1,...,w_m).$$
	Moreover suppose that for a chosen one dimensional subtorus all of the limits $\lim_\sigma t^{w_i}\alpha$
	exist. Then the limit
	$\lim_\sigma t^w\alpha$
	exist as well.
\end{lemma}
The proof  is fairly elementary. We just compare the coefficients of the numerator and denominator of $\alpha$ restricted to one dimensional torus.
\bigskip

\noindent{\it Proof of Lemma \ref{lem:limit-frac}.}
Suppose that $\partial Y$ consists of $m$ components
$D_1,D_2,\dots, D_m$.
For a subset of indices $I \subset\{1,...,m\}$ consider the bundle:
$$ L_I=\prod_{j\in I} \O_Y(D_j)_{|F}$$
Let $D=\sum \lambda_k D_k$ be a $\Q$ divisor. The weight of the bundle {$L=t^{-\we_F(D)}\,\O_Y(\lceil D\rceil)_{|F}$} lies in the convex hull of the weight of the bundles $L_I$.
By the  Lemmas \ref{lem:prostowanie} and \ref{lem:conv} it is enough to show that the limits
\begin{align} \label{wyr:2}
\lim_\sigma
\frac{
	L_I\cdot\mC(Y^o\subset Y)_{|F}}{eu\left(\nu(F\subset Y)\right)}=
\lim_\sigma\frac{\mC\left(Y^o\subset Y\right)_{|F}}{eu\left(\nu(F\subset Y)\right)}\prod_{j\in I} \O_Y(D_j)_{|F}
\end{align}
exist for  a general one parameter subgroup.
The following simple observation allows to remove some factors from the above expression.
\begin{lemma}
	Consider a component $D_j\subset \partial Y$. Let $F \subset Y^\T$ be a component of the fixed point set. Suppose that $F$ is not contained in $D_j$. Then the weight of the bundle $\O_Y(D_j)$ at $F$ is trivial
	$$\we_F(\O_Y(D_j))=0\,.$$
\end{lemma}
\begin{proof}
	The intersection $(Y\setminus D_j)\cap F$ is not empty, hence  in a point not belonging to $D_j$ the weight of $\O_Y(D_j)$ is trivial. The  weight does not depend on the choice of a point in $F$.
\end{proof}

From the above lemma and from Lemma \ref{lem:prostowanie} it follows that when the variety $F$ is not contained in $D_j$ we may omit the bundle $\O_Y(D_j)$ in the expression (\ref{wyr:2}). The existence of the considered limit follows from the next  two following lemmas.
\begin{lemma} \label{lem:SNC}
	Consider a smooth quasiprojective $\T$-variety $Y$. Let $D_1,...,D_m$ be smooth, equivariant, codimension one subvarieties with SNC intersections. Consider a subset \hbox{$I\subset\{1,...,m\}$}. Denote by \begin{itemize}
		\item $Y^o=Y\setminus \bigcup_{{j=1}}^m D_j$,
		\item $D_I=\bigcap_{j\in I} D_j$,  $D_I^o=D_I\setminus \bigcup_{j\notin I} D_j$,
	\end{itemize}
	There is an equality:
	$$\mC\left(Y^o\subset Y\right)_{|D_I} =
	(1+y)^{|I|}\prod_{j \in I} \O_Y(-D_j)_{|D_I}\,\mC\left(D_I^o\subset D_I\right)
	.$$
\end{lemma}
\begin{proof}The proof is an application of
the formula \eqref{obciecie} several times, applied for $Z=Y$ or to intersections of  divisors $D_k$ with $k\not\in I$ .\end{proof}

\begin{corollary}\label{cor:SNC} For a smooth  subvariety $F\subset D_I$ we have
	$$\mC\left(Y^o\subset Y\right)_{|F}\prod_{j \in I} \O_Y(D_j)_{|F}=
	(1+y)^{|I|}\mC\left(D_I^o\subset D_I\right)_{|F}\,.$$
\end{corollary}

Finally we arrive {at} computation of the limit appearing in \eqref{wyr:2}.

\begin{lemma}
	Consider the situation such as in the previous Corollary \ref{cor:SNC}. Suppose that $F$ is a component of the fixed point set $Y^\T$.
	Consider one dimensional subtorus $\sigma:\C^*\to \T$ such that all normal weights to the subvariety $F$ are nonzero.
	Then the limit
	$$\lim_\sigma\frac{\mC\left(Y^o\subset Y\right)_{|F}}{eu(\nu(F\subset Y))}\prod_{j \in I} \O_Y(D_j)_{|F}$$
	exists. If all the weights of the line bundles $\O_Y(D_j)_{|F}$ are negative then the limit is equal to
	$$
	(1+y)^{|I|}\mC\left((D_I^o)_F^{\sigma+}\to F\right)\,,$$
	otherwise  the limit is equal to $0$.
	Here
	$$(D_I^o)_F^{\sigma+}=\{x\in D_I^o\;:\;{\lim_{\tt \to 0}\sigma(\tt)\cdot x\in F}\;\}\,.$$
\end{lemma}

\begin{proof}
	Corollary \ref{cor:SNC}
	implies that
	$$
	\lim_\sigma\frac{\mC\left(Y^o\subset Y\right)_{|F}}{eu(\nu(F\subset Y))}\prod_{j \in I} \O_Y(D_j)_{|F}=
	\lim_\sigma\frac{\mC\left(D_I^o\subset D_{I}\right)_{|F}}{eu(\nu(F\subset D_{I}))}
	(1+y)^{|I|}\prod_{j\in I}\frac1{1-\O_Y(-D_j)_{|F}}\,.
	$$
	Theorem 4.2 from \cite{FRW} (see also theorem 10 from \cite{WeBB} and theorem 4.4 from \cite{K1})
	implies that the limit
	$$\lim_\sigma\frac{\mC\left(D_I^o\subset D_{I}\right)_{|F}}{eu(\nu(F\subset D_{I}))}$$ exists and it is equal to
	$$\mC\left((D_I^o)_F^{\sigma+}\to F\right)\,.$$
	Moreover
	$$ \lim_\sigma\frac{1}{1-\O_Y(-D_j)_{|F}}=
	\begin{cases}
	0& \text{ if the weight of }\O_Y(D_j)_{{|F}}\text{ is positive,} \\
	1& \text{ if the weight of }\O_Y(D_j)_{{|F}}\text{ is negative.}
	\end{cases}
	$$
	Multiplying these  limits we arrive at the desired equality.
\end{proof}

Applying Lemma \ref{lem:FRW} and the decomposition \eqref{wyr:1} we obtain the proof of Theorem \ref{thm:main}.

\section{Choice of the boundary divisor}
\label{sec:dependency}
We consider situation described in section \ref{sec:setup}. We will show that the slope $$s=\L^{1/n}\in Pic(M)\otimes \Q$$ determines a $\Q$-Cartier divisor $\Delta_{\ee_0,s}$ in $X$.

\begin{definition}\rm We say that a meromorphic section  of a linearized bundle $\L_{|X}$ is good if
	\begin{itemize}
		\item the section $v$ is an eigenvector of the torus $\T$,
		\item the section $v$ does not vanish nor has a pole at the center of the cell.
\end{itemize}\end{definition}
For a slope $s=\L^{1/n}$ any good section of $\L$ defines the divisor $\Delta_{\ee_0,s}=\frac1n div(v)$.
The divisor $\Delta_{\ee_0,s}$ is $\T$-invariant, hence  $\Delta_{\ee_0,s}$ has support in $\partial X$.
\medskip

The twisted motivic Chern class $i_*\mC(X,\partial X;\Delta_{\ee_0,s})$ is considered as an element of the K-theory of the ambient space $M$.
The axioms \ref{def:enve} are based on the comparison of this element with the corresponding elements for different choices of the fixed point.
The previous sections imply that  $\mC(X,\partial X;\Delta_{\ee_0,s})$ satisfies the normalization axiom (\ref{def:enve}.2) and the Newton inclusion property (\ref{def:enve}.3), provided that $\Delta_{\ee_0,s}$ is given by a good section of $\L$.
It remains to show that  good sections exist. Moreover we will show that the divisor associated with the slope $s=\L^{1/n}$ is unique, although we already know that stable envelopes are unique.

\subsection{Existence of a good section}
Denote by $\iee$ the inclusion $$\iee: X\hookrightarrow M.$$
Let  $s=\L^{1/n}\in Pic(M)\otimes \Q$ a slope.
By \cite[Lemma 2.14]{Brion}  we can assume that $\L$ admits a linearization and we fix a linearization.

\begin{lemma} There exists a good section. \end{lemma}
\begin{proof}
	First assume that  $\L$ is generated by global sections. The torus $\T$ acts on the global sections $H^0(X;\iee^*\L)$.  Since the eigenvectors span that space we can find a section $v$ such that
	\begin{itemize}
		\item $v\in H^0(X;\iee^*\L)$ is an eigenvector of $\T$,
		\item $v(\ee_0)\neq0$.
	\end{itemize}
	If $\L$ is not globally generated we can tensor it with a suitable power of an equivariant ample bundle $\O(1)$ and  argue as before. {If $v_1$ is a good section of $\L(n)$ and $v_2$ is a good section of $\O(n)$ then $v_1/v_2$ is a good meromorphic  section of $\L$.}
\end{proof}
\begin{remark}\rm For a line bundle $\L$ considered with a fixed linearization and $\O_X(n\Delta_{\ee_0,s})$ with the natural linearization we have
	$$\iee^*\L=t^{\we_{\ee_0}(\L)}\cdot\O_X(n\Delta_{\ee_0,s}) \,.$$
\end{remark}

\subsection{Uniqueness of the boundary divisor}

\begin{lemma}\label{lem:ind}
	Let $v_1$ and $v_2$ be two good meromorphic sections of $\L$ on $X$. Then $v_1$ and $v_2$ are proportional.
\end{lemma}
\begin{proof} The sections $v_i$ are defined on a $\T$-equivariant open neighbourhood $U$ containing $\ee_0$. Let $\chi_i:\T\to\C^*$ be the corresponding characters, i.e.
	$t\cdot v_i=\chi_i(t)v$ for $i=1,2$.
	The quotient $v_1/v_2$ defines a rational function
	$$\theta:=\frac{v_1}{v_2}:X\supset U\to \C\,.$$
	The map $\theta$ is equivariant if we consider the action of the torus $\T$ on $\C$ given by the character $\frac{\chi_1}{\chi_2}$. Moreover $\theta$ does not have zero nor pole at  $\ee_0$. Since $\ee_0$ is a fixed point, $\theta(\ee_0)\neq 0$ is fixed as well. Thus the action of $\T$ on $\C$ is trivial. It follows that characters $\chi_1$ and $\chi_2$ coincide and the map $\theta$ is constant on the orbits of $\T$. The map $\theta$ is defined at $\ee_0$ so it is defined and constant on the whole Bia{\l}ynicki-Birula cell $X^o$.
	It follows that $\theta$ is also defined and constant  on the closure $X$ of the BB-cell $X^o$. Thus $\theta=v_1/v_2$ is constant.
\end{proof}

This finishes construction of the characteristic classes of Bia{\l}ynicki-Birula cells which satisfy the normalization and the Newton  inclusion property \ref{def:enve}(2,3) of Okounkov's stable envelopes in the K-theory.

\section{Support axiom}

\label{sec:support}

We consider the situation described in section \ref{sec:setup}.
The support axiom Def \ref{def:enve}(1)
  \begin{equation}\label{def:enve1}
  Stab^s(\ee_0) \quad\ \text{is supported by} \quad \bigsqcup_{\ee \le \ee_0} (T^*M)^+_{\ee} \,.\end{equation}
was  discussed in \cite[Remark after Theorem 3.1]{RTV-trig}, \cite[Lemma 5.2-4]{RTV} and in \cite{K}.
Suppose that the variety $M$ is projective. Then by \cite{BB2} the relation
$$\ee_1 \preccurlyeq^o \ee_2 \quad \Longleftrightarrow \quad M_{\ee_1}^+\cap\overline{M_{\ee_2}^+}\neq\emptyset $$
generates a partial order $\preccurlyeq$ on the fixed point set.
  We have the support containment
\begin{equation}\label{eq:stupid_zawieranie}
i_*\mC(\overline{M^+_{\ee_0}},\partial M^+_{\ee_0};\Delta_{\ee_0,s})
\quad\text{is supported by}\quad \bigsqcup_{\ee\le \ee_0} M^+_{\ee}\,.\end{equation}
We need a stronger condition for $\pi^*i_*\big(\mC(\overline{M^+_{\ee_0}},\partial M^+_{\ee_0};\Delta_{\ee,s})\big)$ given by the axiom \eqref{def:enve1}.
By lemma \ref{lem:pod} (cf. \cite[Lemma {5.7}]{K}) the condition \eqref{eq:stupid_zawieranie} together with the divisibility condition
\begin{equation}\label{eq:podzielnosc}
\mC(M^+_{\ee})_{|\ee}\quad \text{divides}\quad i_*\mC(\overline{M^+_{\ee_0}},\partial M^+_{\ee_0};\Delta_{\ee_0,s})_{|\ee}\end{equation}
implies the axiom \eqref{def:enve1}.
\medskip

Suppose there exists a neighbourhood of $\ee$ such that the BB-decomposition is of a product form (see \cite[Def 3.10]{K}). Then the divisibility condition holds. Although the assumption that the stratification is locally a product is a strong demand, there are important examples of spaces having that property. It is enough to mention the homogeneous spaces $G/P$ where the BB-cells are the Schubert cells, see \cite[Theorem {9.1}]{K}.
By Proposition \ref{prop:prod} the divisibility axiom holds. This way we have proven the Corollary given in the introduction.
\medskip

A weaker condition guaranteeing the divisibility condition is existence of a sufficiently transverse slice. A natural candidate for a slice is the minus-cell
$M_{\ee}^-$, see \eqref{eq:minus}. If the BB-decomposition is a stratification satisfying the Whitney condition, then $M_{\ee}^-$ is transverse not only to $M_{\ee}^+$ but also to the neighbouring strata. It is not clear whether Whitney conditions imply divisibility of twisted motivic Chern classes. A stronger notion of transversality would suffice. It is called in \cite[Definition 8.4]{FRW2}  the motivically transversality.
Recall that for a resolution $Y\to X$ with a SNC divisor $E=\bigcup_{i\in I} E_i\subset Y$ and for any subset $J\subset I$  the intersection $\bigcap_{i\in J}E_i$ is denoted by $E_J$.

\begin{proposition}Suppose  that
the map $f:Y\to M$ is  a resolution  of
$(\overline{M_{\ee_0}^+},\partial \overline{M_{\ee_0}^+})$. Let $E=f^{-1}(\partial \overline{M_{\ee_0}^+})$. Assume that $f$ restricted to the intersection   $E_J$ is transverse to $M^-_\ee$
$$f_{|E_J}
\;\pitchfork\;M^-_\ee$$
for any $J\subset I$.
Then
\begin{equation}\mC(M^-_\ee)_{|\ee}\quad\text{divides}\quad
\mC(M^+_{\ee_0},\partial M^+_{\ee_0};\Delta_{\ee_0,s})_{|\ee}\,.
\end{equation}
\end{proposition}

\begin{proof}For $J\subset I$ let $f_J=f_{|E_J}$. Set $\widetilde M^-_\ee= f^{-1}_J(M^-_\ee)$ and
denote by $g_J$ the restriction of $f_J$ to $\widetilde M_\ee^-\cap E_J$. Let $\iota_J$, $\tilde \iota_J$ be the inclusions
$$\begin{matrix}\widetilde M_\ee^-\cap E_J&\stackrel{\widetilde\iota_J}\longrightarrow&E_J\\\\
^{g_J}\big\downarrow\phantom{^{g_J}}&&\phantom{{}^{f_J}}\big\downarrow{}^{f_J} \\\\
M^-_\ee&\stackrel{\iota_J}\longrightarrow&M
\end{matrix}$$
The transversality condition implies, that $\widetilde M^-_\ee$ is smooth and
$$\widetilde\iota^*_J(T E_J)=T(\widetilde M^-_\ee\cap E_J) \oplus g^*_J(\nu(M^-_\ee\subset M))\,.$$
Let $\widetilde\L=\O_Y(\lceil(f^*(\Delta_{e_0,s})\rceil))$ and $M_J=M_\ee^-\cap E_J$.
Again, by transversality $\iota_J^*\circ f_{J*}=g_{J*}\circ\tilde \iota^*_J$, hence
\begin{align*}\iota_J^*f_{J*}(\widetilde\L\cdot\lambda_y(T^*E_J))
&=g_{J*}\tilde \iota^*_J(\widetilde\L\cdot\lambda_y(T^*E_J))\\
&=g_{J*}\big(\tilde \iota^*_J\widetilde\L\cdot\lambda_y(T^*\widetilde M_J)\cdot g_J^*(\lambda_y\nu^*(M^-_\ee\subset M))\big)\\
&=g_{J*}\big(\tilde \iota^*_J\widetilde\L\cdot\lambda_y(T^*\widetilde M_J)\big)\cdot \lambda_y\nu^*(M^-_\ee\subset M)\,.
\end{align*}
At the point $\ee$ the normal bundle $\nu(M^-_\ee\subset M)$ is equal to $T_\ee M^+_\ee$. We conclude that
$$f_{J*}(\widetilde\L\cdot\lambda_y(T^*E_J))_\ee\quad\text{is divisible by}\quad \lambda_y(T^*_\ee M^+_\ee)\,.$$
By the formula \eqref{def:mc} combined with the definition of the twisted classes Def.~\ref{def-twisted} we obtain the divisibility.\end{proof}
\medskip

To sum-up: the divisibility condition holds provided that the BB-decomposition enjoys certain equisingularity property. The recent results of J.~Sch{\"u}rmann \cite{Sch} about non-characteristic pull-back suggest that the Whitney conditions would be enough. That is so for untwisted motivic Chern class. The results of Sch{\"u}rmann imply that the divisibility condition holds for negative slopes which are sufficiently close to zero, i.e. when $\lceil f^*(\Delta_{\ee_0,s})\rceil=0$.

\end{document}